\documentclass[11pt]{article}
\usepackage{amsmath,latexsym,amsxtra,enumerate,
color, cancel}
\usepackage{bbm, dsfont}

\usepackage{amsthm}
\usepackage[english]{babel}
\usepackage[usenames,dvipsnames,svgnames,table]{xcolor}
\usepackage{times, amsfonts, amssymb, amsthm, amscd, graphicx,stmaryrd}
\usepackage[mathscr]{euscript}

\RequirePackage[colorlinks,citecolor=blue,urlcolor=blue]{hyperref}
\hypersetup{colorlinks=true}
\usepackage[active]{srcltx}
\usepackage{a4wide}

\newcommand{\R}{\mathbb{R}}

\newcommand{\N}{\mathbb{N}}

\theoremstyle{plain}
\newtheorem{thm}{Theorem}[section]
\newtheorem{lemma}[thm]{Lemma}

\theoremstyle{definition}

\newtheorem{definition}[thm]{Definition}

\newtheorem{rem}[thm]{Remark}

\newtheorem{example}[thm]{Example}

\allowdisplaybreaks

\newcommand{\equa}{\begin{eqnarray*}}
\newcommand{\tion}{\end{eqnarray*}}
\newcommand{\equal}{\begin{eqnarray}}
\newcommand{\tionl}{\end{eqnarray}}

\makeatletter
\def\timenow{\@tempcnta\time
\@tempcntb\@tempcnta
\divide\@tempcntb60
\ifnum10>\@tempcntb0\fi\number\@tempcntb
:\multiply\@tempcntb60
\advance\@tempcnta-\@tempcntb
\ifnum10>\@tempcnta0\fi\number\@tempcnta}
\makeatother

\title{Complements and Improvements Regarding Distributivity of the Product for $\sigma$-Algebras with Respect to the Intersection}

\author{K.P.S.~Bhaskara Rao$^1$, Alexander Steinicke$^2$\\
}

\date{}
\begin{document}

\maketitle
\begin{abstract}\noindent
We present a variety of refined conditions for $\sigma$ algebras $\mathcal{A}$ (on a set $X$), $\mathcal{F}, \mathcal{G}$ (on a set $U$) such that the distributivity equation
$$(\mathcal{A}\otimes\mathcal{F})\cap(\mathcal{A}\otimes\mathcal{G})=\mathcal{A}\otimes\left(\mathcal{F}\cap\mathcal{G}\right),$$
holds -- or is violated. \\
The article generalizes the results in \cite{Steinicke21} and includes a positive result for $\sigma$ algebras generated by at most countable partitions, was not covered before. We also present a proof that counterexamples may be constructed whenever $X$ is uncountable and there exist two $\sigma$-algebras on $X$ which are both countably separated, but their intersection is not. We present examples of such structures. In the last section, we extend \cite[Theorem 2]{Steinicke21} from analytic to the setting of Blackwell spaces.
\end{abstract}

\vspace{1em}
{\noindent \textit{Keywords:} sigma algebra; intersection of sigma algebras; product sigma algebras; counterexample for sigma algebras
}
{\noindent
\footnotetext[1]{Department of Computer Information Systems, Indiana University Northwest, Gary, IN, USA. \\
\hspace*{1.5em}  {\tt bkoppart{\rm@}iu.edu }}
\footnotetext[2]{Department of Mathematics and Information Technology, Montanuniversitaet Leoben, Austria. \\ \hspace*{1.5em}  {\tt alexander.steinicke{\rm@}unileoben.ac.at}}}

%%%%%%%%%%%%%%%%%%%%%%%%%%%%%%%%%%%%%%%%%%%%%%%%%%%%%%%%%%%%%%%%%%%%%%%%%%%%%%%%%%%

\section{Introduction}%\phantomsection

The question of distributivity of $\sigma$-algebras with respect to the intersection has recently been studied in \cite{Steinicke21}, motivated by questions from stochastic analysis (therein and in \cite{Steinicke16}), but also coming from a question about sequences of probability spaces in general stochastics \cite{Parry, stackexchange}$^3$\footnotetext[3]{There was an erroneous statement in \cite{Steinicke21} regarding the discussions in the references \cite{Parry,stackexchange}:\\ The assertion that $(\mu\otimes\nu)(K)=0$ for every $K\in \bigcap {n\in \N}(\mathcal{A} n\otimes\mathcal{U})\setminus \biggl(\bigcap {n\in \N}\mathcal{A} n\biggr)\otimes\mathcal{U}$ has to be replaced with: For all $K\in \bigcap {n\in \N}(\mathcal{A} n\otimes\mathcal{U})$ there is an $L\in \biggl(\bigcap {n\in \N}\mathcal{A} n\biggr)\otimes\mathcal{U}$ such that $(\mu\otimes\nu)(L\triangle K)=0$.}. While the problem can be formulated quite simply using only well known basic constructions, its answer still is nontrivial. The question, apart from the investigation in \cite{Steinicke21} has not been addressed in various remarkable sources for results on $\sigma$-algebras (sometimes referred to as {\it Borel structures}) exceeding standard literature, such as Aumann \cite{Aumann}, Basu \cite{Basu}, Bhaskara Rao and Bhaskara Rao \cite{BhaskaraBhaskara}, Bhaskara Rao and Rao \cite{RaoRao}, Bhaskara Rao and Shortt \cite{RaoShortt}, Blackwell \cite{Blackwell}, Georgiou \cite{Georgiou}, Grzegorek \cite{Grzegorek}, Rao \cite{Rao} to mention an (incomplete) list of contributions to the theory. A precise description of the problem is the following.\medskip

Let $\mathcal{A}$ be a $\sigma$-algebra on a (nonempty) set $X$ and $\mathcal{F},\mathcal{G}$ be two $\sigma$-algebras on a (nonempty) set $U$. The product of the $\sigma$-algebras $\mathcal{C}$ on a set $Y$ and $\mathcal{D}$ on a set $Z$ is denoted by $\mathcal{C}\otimes\mathcal{D}$ on $Y\times Z$ and  is defined as the smallest $\sigma$-algebra containing all Cartesian products (or rectangles) $\{C\times D: C\in \mathcal{C}, D\in \mathcal{D}\}$. We ask, for which $\sigma$-algebras $\mathcal{A},\mathcal{F}, \mathcal{G}$ is
\begin{align}\label{eq:main}
(\mathcal{A}\otimes\mathcal{F})\cap(\mathcal{A}\otimes\mathcal{G})=\mathcal{A}\otimes\left(\mathcal{F}\cap\mathcal{G}\right),
\end{align}
which means that '$\otimes$' is distributive with respect to '$\cap$'.

The first trivial observation is the inclusion 
\begin{align*}
(\mathcal{A}\otimes\mathcal{F})\cap(\mathcal{A}\otimes\mathcal{G})\supseteq\mathcal{A}\otimes\left(\mathcal{F}\cap\mathcal{G}\right).
\end{align*}
Also rather easy to show is the relation
\begin{align*}
(\mathcal{A}\otimes\mathcal{F})\vee(\mathcal{A}\otimes\mathcal{G})=\mathcal{A}\otimes\left(\mathcal{F}\vee\mathcal{G}\right),
\end{align*}
see e.g. \cite[Proof of Lemma 3.2, Step 2]{Steinicke16}.
The nontrivial results in \cite{Steinicke21} so far pointed out a counterexample to \eqref{eq:main} (Theorem 1) and showed that in case of analytic measurable spaces $(X,\mathcal{M}), (U,\mathcal{U})$ with $\mathcal{A}\subseteq\mathcal{M}$, $\mathcal{F}, \mathcal{G}\subseteq\mathcal{U}$,  and countably generated $\sigma$-Algebras $\mathcal{F}\cap\mathcal{G}, \mathcal{A}, (\mathcal{A}\otimes\mathcal{F})\cap(\mathcal{A}\otimes\mathcal{G})$, the equation \eqref{eq:main} is equivalent to all atoms of $(\mathcal{A}\otimes\mathcal{F})\cap(\mathcal{A}\otimes\mathcal{G})$ being products (Theorem 2).\medskip

Here, we generalize these results in the following directions:
\begin{itemize}
\item We show that the distributivity equation \eqref{eq:main} holds if either one of the three $\sigma$-algebras are given by a countable partition or if  $\mathcal{F}\cap\mathcal{G}$ is given by a countable partition. This is not covered in \cite{Steinicke21}.
\item We show that counterexamples are not limited to the special one from \cite[Theorem 1]{Steinicke21} but can be constructed whenever $X$ is not countable and $\mathcal{F}, \mathcal{G}$ are countably separated, but $\mathcal{F}\cap\mathcal{G}$ is not.
\item We obtain a characterization of the distributivity equation \eqref{eq:main} in terms of its atoms, similar to \cite[Theorem 2]{Steinicke21}, but now for strongly Blackwell spaces instead of analytic ones.
\end{itemize}

\section{Distributivity for $\sigma$-algebras given by countable partitions}

In this section we will treat the countable partition case, as the finite case basically works in a similar, but easier way. We will write $\sigma(\{C_1, C_2,\dotsc\}$ for the $\sigma$-algebra generated by $\{C_1, C_2, \dotsc\}.$  Let $\N$ stand for the set of  positive integers. We will say that a $\sigma$-algebra $\mathcal{C}$ is given by a countable partition if there are countably many(or finitely many) pairwise disjoint nonempty sets $\{C_1, C_2, \dotsc\}$ in $\mathcal{C}$ such that $\{\bigcup_{i \in I} C_i: I \subset \N\} =\mathcal{C} .$ Note that on a countable set, every $\sigma$-algebra is given by a countable partition. Observe also that if $\mathcal{C}$ is a $\sigma$-algebra given by a countable partition $\{C_1, C_2, \dotsc\}$ and $\mathcal{D}$ is another $\sigma$-algebra  then every set in $\mathcal{C} \otimes \mathcal{D}$ is of the form $\bigcup_{i \in \N} (C_i \times D_i)$ for some sequence $(D_i)_{i\in \N}$ such that $D_i \in \mathcal{D}$  for all $i \geq 1.$ This representation is also unique. We will be using this result repeatedly in some of the proofs below.
\begin{thm}
Let $\mathcal{A}$ be a $\sigma$-algebra given by a countable partition of $X$. Let $\mathcal{F}, \mathcal{G}$ be $\sigma$-algebras on $U$. Then equation \eqref{eq:main} is true.
\end{thm}

\begin{proof} Let $\mathcal{A}$ be given by a countable partition $\{A_1, A_2, \dotsc\}.$

We will take a general set $B \in (\mathcal{A}\otimes\mathcal{F})\cap (\mathcal{A}\otimes\mathcal{G})$ and show that $B \in \mathcal{A}\otimes\left(\mathcal{F}\cap\mathcal{G}\right).$ By the above mentioned result, $B$ can be written as $\bigcup_{i \in \N} (A_i \times F_i)$ for some $F_i \in \mathcal{F}$ for every $i.$ Since $B \in \mathcal{A}\otimes\mathcal{G}$, by taking an $x$ in $A_i$ and considering the section $B_x = F_i$, we see that $F_i \in \mathcal{G}$ for every $i.$  This shows that $B \in \mathcal{A}\otimes\left(\mathcal{F}\cap\mathcal{G}\right).$
\end{proof}

The following result concerns the right factors.

\begin{thm}
If at least one of $\mathcal{F}$ and $\mathcal{G}$ is given by a countable partition, equation \eqref{eq:main} is true.  More generally, if $\mathcal{F} \cap \mathcal{G}$ is given by a countable partition, equation \eqref{eq:main} is true.  
\end{thm}
\begin{proof} Clearly, if $\mathcal{F}$ (or $\mathcal{G}$) is given by a countable partition,  $\mathcal{F} \cap \mathcal{G}$ is also given by a countable partition. Hence it is sufficient to prove the second part of the statement of the theorem. Let $\mathcal{F} \cap \mathcal{G}$ be given by a countable partition $\{H_1, H_2, \dotsc\}.$

We will take a general set $B \in (\mathcal{A}\otimes\mathcal{F})\cap (\mathcal{A}\otimes\mathcal{G})$ and show that $B \in \mathcal{A}\otimes\left(\mathcal{F}\cap\mathcal{G}\right).$ Let $x$ be a point in $X.$ Consider the $x$-section $B_x = \{u:(x, u) \in B\}.$ Then, $B_x \in \mathcal{F} \cap \mathcal{G}.$ Hence $B_x \supseteq H_i$ for some $i$ if $B_x \neq \emptyset.$ Now, for a particular $i \in \N,$ if we define the set $A_i = \{x: B_x \supseteq H_i\}$ then, $A_i \in \mathcal{A}$ and $B \supseteq A_i\times H_i.$ It easily follows that $B = \bigcup_{i\in \N} (A_i \times H_i).$ Thus $B \in \mathcal{A}\otimes\left(\mathcal{F}\cap\mathcal{G}\right).$

 %By the result mentioned at the beginning of the section, $B$ can be written as $\bigcup_{i \in \N} (A_i \times F_i)$ where $ A_i \in \mathcal{A}$ for every $i.$ Since $B \in \mathcal{A}\otimes\mathcal{G}$ we see that $F i \in \mathcal{G}$ for every $i.$  This shows that $B \in \mathcal{A}\otimes\left(\mathcal{F}\cap\mathcal{G}\right).$

\end{proof}

Thus we have seen that if at least one of the three $\sigma$-algebras $\cal A, \cal F, \cal G$ in equation \eqref{eq:main} is given by a countable partition, \eqref{eq:main} is true.  We will extend this result to the case of $\sigma$-algebras with the property that every countably generated sub-$\sigma$-algebra is given by a countable partition.  The $\sigma$-algebra of countable and  co-countable sets on any infinite set is an example of such a $\sigma$-algebra (which itself is not even countably generated). This $\sigma$-algebra is atomic. There are examples of atomless $\sigma$-algebras (on any uncountable set) with this property. See \cite[Remark 5, p.~108]{RaoRao1} for examples of such $\sigma$-algebras.

\begin{thm}\label{thm:countable-sub-sigma}
If one of the three $\sigma$-algebras  $\mathcal{A}$, $\mathcal{F}$ and $\mathcal{G}$ in \eqref{eq:main} is such that   every countably generated sub-$\sigma$-algebra is given by a countable partition, then \eqref{eq:main} is true. However, if $\mathcal{F} \cap \mathcal{G}$ has the property that every countably generated sub-$\sigma$-algebra is given by a countable partition, but $\mathcal{F} \cap \mathcal{G}$ itself is not given by a countable partition, \eqref{eq:main} need not be true.
\end{thm}

\begin{proof} We will treat the case of $\mathcal{F}$  having the property that every countably generated sub-$\sigma$-algebra is given by a countable  partition $\{F_1, F_2, \dotsc\}.$ Other cases can be covered using a similar argument.

We will take a general set $B \in (\mathcal{A}\otimes\mathcal{F})\cap (\mathcal{A}\otimes\mathcal{G})$ and show that $B \in \mathcal{A}\otimes\left(\mathcal{F}\cap\mathcal{G}\right).$ 

We know that if a set $B \in \mathcal{C}\otimes\mathcal{D}$ there exists $\{C_1, C_2, \dotsc \}$ in $\mathcal{C}$ and $\{D_1, D_2, \dotsc \}$ in $\mathcal{D}$ such that $B \in 
\mathcal{C}_0 \otimes \mathcal{D}_0$ where $\mathcal{C}_0$ is the sub-$\sigma$-algebra generated by $\{C_1, C_2, \dotsc \}$  and $\mathcal{D}_0$ is the sub-$\sigma$-algebra generated by $\{D_1, D_2, \dotsc \}$. We will apply this result to $\mathcal{A}\otimes\mathcal{F}.$

Since $B \in \mathcal{A}\otimes\mathcal{F}$ there exists $\{A_1, A_2, \dotsc \}$ in $\mathcal{A}$ and $\{F_1, F_2, \dotsc \}$ in $\mathcal{F}$ such that $B \in 
\mathcal{A}_0 \otimes \mathcal{F}_0$ where $\mathcal{A}_0$ is the sub-$\sigma$-algebra generated by $\{A_1, A_2, \dotsc \}$  and $\mathcal{F}_0$ is the sub-$\sigma$-algebra generated by $\{F_1, F_2, \dotsc \}.$ Since $\mathcal{F}_0$ is a countably generated sub-$\sigma$-algebra it is given by a countable partition. The previous theorem gives us that $B \in \mathcal{A}_0\otimes\left(\mathcal{F}_0\cap\mathcal{G}_0\right).$  Hence $B \in \mathcal{A}\otimes\left(\mathcal{F}\cap\mathcal{G}\right).$

The example in \cite[Theorem 1]{Steinicke21} has the property that $\mathcal{F} \cap \mathcal{G}$ is the countable co-countable $\sigma$-algebra and this $\sigma$-algebra has the property that every countably generated sub $\sigma$-algebra is given by a countable partition. Thus we have concluded the second statement of the theorem.

\end{proof}

In Theorem \ref{thm:countablySeparatedSub} below, we will show that the second part of the above theorem is true in general.

We remark that, as a consequence of Theorem 2.2,  if $\cal F \cap \cal G$ is the trivial $\sigma$-algebra $\{\emptyset, U\}$ or a finite $\sigma$-algebra, then also \eqref{eq:main} is true.

\section{Counterexamples for Uncountable Sets}

In this section here, we first point out, that for $\sigma$-algebras $\mathcal{F}$ and $\mathcal{G}$ on a set $U$, which are not given by countable partitions, such that $\mathcal{F}\cap\mathcal{G}$ is infinite and neither $\mathcal{F}\subseteq\mathcal{G}$ nor $\mathcal{G}\subseteq\mathcal{F}$ hold, and for a $\sigma$-algebra $\mathcal{A}$, there are examples such that  \eqref{eq:main} does not hold. We already saw classes of examples of $\sigma$-algebras such that \eqref{eq:main} holds.

Let us recall a basic definition first.
\begin{definition}
Let $(Y,\mathcal{C})$ be a measurable space. 
\begin{enumerate}[(i)]
\item The $\sigma$-algebra $C$ is called {\it separated} if there is a set $I$, together with sets $\{A_i:i\in I\}\subseteq\mathcal{C}$, that separate the points of $Y$, that is, for any two points $x,y\in Y$ there is an $i\in I$ such that $x\in A_i$ and $y\notin A_i$. We call such a system of sets {\it separator}.
\item The $\sigma$-algebra $\mathcal{C}$ is called {\it countably separated} if $\mathcal{C}$ contains a countable separator. 
\end{enumerate}
\end{definition}
\begin{rem}
\begin{enumerate}[(i)]
\item Note that if a countable separator exists for $(Y,\mathcal{C})$ then $\cal C$ contains all singletons: just write for $x\in Y$,
$$\{x\}=\bigcap_{\substack{i\geq 1\\ x\in A_i}}A_i\cap\bigcap_{\substack{i\geq 1\\ x\notin A_i}}A_i^c.$$
\item If $\mathcal{C}$ is separated and countably generated by the generator $\{A_i:i\geq 1\}$, then this generator features as countable separator as well: Assume the contrary, that there are no sets separating the points $x$ and $y\in Y$. That means that  all sets $A_i$ contain either both, $x$ and $y$, or none of them. The smallest $\sigma$-algebra that contains all the $A_i$ and which does not separate $x$ and $y$ is the $\sigma$-algebra generated by $\left\{A_i: i\geq 1\right\}$ and must thus be $\mathcal{C}$. But $\mathcal{C}$ is separable. Hence $\left\{A_i: i\geq 1\right\}$ is a separator.
\end{enumerate}
\end{rem}
\begin{example}
\begin{enumerate}[(a)]
\item The Borel $\sigma$-algebra of any separable metric space is countably generated and countably separated.
\item The $\sigma$-algebra of Borel sets of $\R$ invariant of translation by $1$ (or any other nonzero number) is countably generated but not countably separated (any set containing $1$ also contains $\mathbb{Z}$).
\item The $\sigma$-algebra generated by the analytic sets of $[0,1]$ is not countably generated \cite[p.~15]{RaoRao} but countably separated (just take the separator of the Borel sets).
\item The $\sigma$-algebra of countable and co-countable sets on an uncountable set $Y$ is neither countably generated nor countably separated.
\end{enumerate}
\end{example}

The next theorem is a (substantial) extension of \cite[Theorem 1]{Steinicke21} in the sense that we provide classes of examples for which \eqref{eq:main} does not hold.
\begin{thm}\label{thm:BcapDcounterexample}
Let $\mathcal{B}$ and $\mathcal{D}$ be any countably separated $\sigma$-algebras on an uncountable set $X$, such that
$\mathcal{B}\cap\mathcal{D}=:\mathcal{C}$ is not countably separated.

Then,
\begin{align}\label{eq:negmain}
((\mathcal{D}\vee\mathcal{B})\otimes \mathcal{B})\cap((\mathcal{D}\vee\mathcal{B})\otimes \mathcal{D})\neq (\mathcal{D}\vee\mathcal{B})\otimes(\mathcal{D}\cap\mathcal{B}).
\end{align}
\end{thm}
\begin{proof}
We will  first show that  the diagonal $\Delta=\{(x,x): x\in X\}$ is in in the left hand side of \eqref{eq:negmain},
Let $\left\{A_i: i\geq 1\right\}$ be a countable separator of $\mathcal{B}$. 
 $\Delta$ can be expressed as
\begin{align*}
\Delta=\bigcap_{i\geq 1}\big((A_i\times A_i) \cup (A_i^c\times A_i^c)\big).
\end{align*}  
This can be seen, as for all $x\in X$ and $i\geq 1$, $(x,x)$ is either contained in $A_i\times A_i$ or $A_i^c\times A_i^c$. To exclude any $(x,y)$ with $x\neq y$, take an $A_i$ separating $x$ and $y$, s.t. say $x\in A_i$ and $y\in A_i^c$. Now $(x,y)\notin (A_i\times A_i) \cup (A_i^c\times A_i^c)$. It follows that $\Delta\in \mathcal{B}\otimes \mathcal{B}$. In the same way, $\Delta \in \mathcal{D}\otimes \mathcal{D}$, showing that $\Delta$ is contained in the left hand side of \eqref{eq:negmain}.\smallskip

If for $\sigma$-algebras $\mathcal{H}, \mathcal{I}$ on $X$, the diagonal $\Delta$ is contained in $\mathcal{H}\otimes\mathcal{I}$, then $\mathcal{I}$ needs to be countably separated. This is a special case e.g. of \cite[Proposition 2.1]{Musial} taking $f$ equal to the identity function.

Since $\mathcal{C}$ is not countably separated, it follows that $\Delta$ cannot be contained in the right hand side of \eqref{eq:negmain}.
\end{proof}

Two examples of pairs of $\sigma$-algebras satisfying the conditions of the above theorem can be found in \cite{RaoRao}. The first such example is from \cite[p.~16]{RaoRao} and the second from \cite[Proposition 57, p.~55]{RaoRao}.

%In all these examples   $\mathcal{B}\cap \mathcal{D}=\left\{B\subseteq X: B\text{ or }X\setminus B\text{ is countable}\right\}$ which is not countably separated.

\begin{example}
\begin{enumerate}[(a)]
\item This example was first published in \cite{Aumann}, following an idea from P.~R.~Halmos, and was also used as a counterexample for \eqref{eq:main} in \cite{Steinicke21}.

The $\sigma$-algebra $\mathcal{D}$ is the preimage of the Borel sets $\mathcal{B}$ on $X=[0,1]$ under a certain function $f\colon{[0,1]}\to{[0,1]}$ constructed  as follows (see also \cite{Aumann}, \cite{Rao}):

Let $\omega_\mathfrak{c}$ be the first ordinal corresponding to the cardinal $\mathfrak{c}$ of the continuum. Let $(M_{\alpha})_{1\leq\alpha<\omega_{\mathfrak{c}}}$ be an enumeration of all uncountable Borel subsets of $[0,1]$ with uncountable complement. Since all uncountable Borel sets have cardinality $\mathfrak{c}$ (see e.g. \cite[Theorem 13.6]{Kechris}) we can associate to each ordinal $\alpha<\omega_\mathfrak{c}$ a triplet $(x_\alpha,y_\alpha,z_\alpha)$ such that $x_\alpha, y_\alpha\in M_\alpha$, $z_\alpha\in [0,1]\setminus M_\alpha$ and $\{x_\alpha,y_\alpha,z_\alpha\}\cap \bigcup_{\beta<\alpha}\left\{x_\beta,y_\beta,z_\beta\right\}=\emptyset$ (as the cardinality of $\bigcup_{\beta<\alpha}\left\{x_\beta,y_\beta,z_\beta\right\}$ is strictly smaller than $\mathfrak{c}$).
Define $f$ as the function that for each $\alpha<\omega_\mathfrak{c}$ maps $x_\alpha\mapsto z_\alpha$, $z_\alpha\mapsto x_\alpha$ and keeps $y_\alpha$ and all points outside $\bigcup_{\beta<\omega_\mathfrak{c}}\left\{x_\beta,y_\beta,z_\beta\right\}$ fixed.  

Finally, set \begin{align*}
\mathcal{D}:=\sigma(f)=\left\{f^{-1}(B):B\in \mathcal{B}\right\}.
\end{align*}
The intersection $\mathcal{B}\cap\mathcal{D}$ is the $\sigma$-algebra of countable and co-countable sets.

\item \cite[Proposition 57]{RaoRao} Let $f$ be a bijection of $[0,1]$ to some analytic, non-Borel set $A\subseteq[0,1]$. Let $\mathcal{B}=\mathcal{B}([0,1])$, and let $\mathcal{B}_A$ be the trace $\sigma$-algebra on $A$. Set $\mathcal{D}=f^{-1}(\mathcal{B}_A)$. Assume that $\mathcal{B}\cap\mathcal{D}$ were countably separated. Then it contains a separator $\{A_i:i\geq 1\}$ and $\sigma(A_i:i\geq 1)$ is countably generated and countably separated. It follows through \cite[Proposition 5]{RaoRao} (the Blackwell nature of $([0,1],\mathcal{B})$ and $([0,1],\mathcal{B})$) that $\sigma(A_i:i\geq 1)=\mathcal{B}=\mathcal{D}$ and that $f$ were an isomorphism $([0,1],\mathcal{B}([0,1]))\leftrightarrow(A,\mathcal{B}_A)$. This would imply that $A$ is a standard Borel set, which it is not. Hence $\mathcal{B}\cap\mathcal{D}$ is not countably separated.
\end{enumerate}
\end{example}

A consequence of Theorem \ref{thm:BcapDcounterexample} is that the second part of Theorem  2.3 is in fact true in general if $\mathcal{F}$ and $\mathcal{G}$ are countably separated.

\begin{thm}\label{thm:countablySeparatedSub}
If $\mathcal{F}$ and $\mathcal{G}$ are countably separated and 
if $\mathcal{F} \cap \mathcal{G}$ has the property that every countably generated sub $\sigma$-algebra is given by a countable partition, but $\mathcal{F} \cap \mathcal{G}$ itself is not given by a countable partition, \eqref{eq:main} is not true.  
\end{thm}
\begin{proof}
If $\mathcal{F} \cap \mathcal{G}$ has the mentioned property, then, we show that  $\mathcal{F} \cap \mathcal{G}$ is not countably separated. If it is, let $\{A_1, A_2, \cdots \}$ be a countable separator. Then $\sigma(\{A_1, A_2, \cdots \})$ being a countably generated sub $\sigma$-algebra is  given by a countable partition. If $\{B_1, B_2, \cdots\}$ is the countable partition, all the $B_i$s need to be singleton sets. This implies that  $ \sigma(\{B_1, B_2, \cdots \})= \sigma(\{A_1, A_2, \cdots \}) = \mathcal{F} \cap \mathcal{G}$ and that  $\mathcal{F} \cap \mathcal{G}$ is given by a countable partition and this is a contradiction. Thus, $\mathcal{F} \cap \mathcal{G}$ is not countably separated. From Theorem \ref{thm:BcapDcounterexample} the result follows.
\end{proof}

\section{Distributivity and Blackwell spaces}

\begin{definition}\label{def:atom}
Let $(Y,\mathcal{C})$ be a measurable space. 
\begin{enumerate}[(i)]
\item An {\it atom} $K\in \mathcal{C}$ is a nonempty set, such that no proper nonempty subset is contained in $\cal C$. In other words, they are the minimal elements with respect to `$\subseteq$' in $\mathcal{C}$. 

\item A $\sigma$-algebra $\mathcal{C}$ is called {\it atomic} if $Y$ is the union of the atoms of $\cal C$. 
\end{enumerate}
\end{definition}

This definition of an atom does not coincide with the similar notion from measure theory, where it is a set of positive measure, such that all of its proper subsets in the $\sigma$-algebra have zero measure. There are $\sigma$-algebras containing atoms and atomless $\sigma$-algebras \cite[Chapter 3]{RaoRao}. Note that if two atoms of $\cal C$, when they exist, are disjoint.\smallskip

\begin{definition}\label{def:analytic}
Let $(Y,\mathcal{C})$ be a measurable space. 
\begin{enumerate}[(i)]
\item A subset of $[0,1]$ is called {\it analytic} if it is the image of a Polish space under a continuous function.

\item We call $(Y,\mathcal{C})$ {\it analytic} if $Y$ is isomorphic (i.e. there is a bijective, bimeasurable function) to an analytic subset of the unit interval and $\mathcal{C}$ is countably generated and contains all singletons of $Y$.

\end{enumerate}
\end{definition}

We cite the following result of Blackwell \cite{Blackwell} and Mackey \cite{Mackey}, see also  \cite[Chapter 2]{RaoRao}.

\begin{lemma}[{\cite[Section 4]{Blackwell}, \cite[Section 4]{Mackey}, \cite[Proposition 6]{RaoRao}}]\label{lem:blackwell}\ \\
If $(Y,\mathcal{C})$ is an analytic space and $\mathcal{W},\mathcal{V}$ are countably generated sub $\sigma$-algebras of $\mathcal{C}$ with the same atoms, then $\mathcal{W}=\mathcal{V}$. 
\end{lemma}

To generalize the situation of Lemma \ref{lem:blackwell}, we recall the definition of a Blackwell space:

\begin{definition}\label{def:blackwell}
Let $(Y,\mathcal{C})$ be a countably generated and separated (and thus countably separated) measurable space.

\begin{enumerate}[(i)]
\item The space $(Y,\mathcal{C})$ is called {\it Blackwell} if the only separated sub-$\sigma$-algebra of $\mathcal{C}$ is itself.
\item The space $(Y,\mathcal{C})$ is called {\it strongly Blackwell} if any two countably generated sub-$\sigma$-algebras with the same atoms coincide.
\end{enumerate}
\end{definition}

For a thorough account on Blackwell spaces see \cite{RaoShortt}.  Therein it is also shown that under the continuum hypothesis (CH) there are Blackwell spaces which are not strongly Blackwell. The same holds if Martin's axiom and ($\neg$\,CH) is assumed.\medskip

%Now, Lemma \ref{lem:atomlem}, together with Definition \ref{def:blackwell} yield
%
%\begin{thm}
%If $(X,\mathcal{M})$ and $(U,\mathcal{U})$ are strongly Blackwell, $\mathcal{A}\subseteq\mathcal{M}, \mathcal{F},\mathcal{G}\subseteq\mathcal{U}$ are all atomic and all products involved in \eqref{eq:main} are Blackwell, then \eqref{eq:main} holds.
%\end{thm}

To obtain a distributivity characterization in the sense of \cite[Theorem 2]{Steinicke21} in the case of Blackwell instead of analytic spaces, we formulate the following Lemma (which has been shown in several places in the literature during proofs, such as in \cite{RaoRao} or \cite{Steinicke21}).

\begin{lemma}\label{lem:productAtoms}
If $\cal B$ on $X$ and $\cal C$ on $Y$ are atomic $\sigma$- algebras, then $\cal B \otimes \cal C$ is atomic and atoms of $\cal B \otimes \cal C$ are cartesian products.
\end{lemma} 
\begin{proof}
If $B \in \cal B$ and $C \in \cal C$  are atoms, to show that $B \times C$ is an atom, observe that every rectangle in $\cal B \otimes \cal C$ either contains $B \times C$ or is disjoint with $B \times C.$ On the other hand, $\{A \in {\cal B} \otimes {\cal C}: A \mbox{ either contains } B \times C \mbox { or is disjoint with } B \times C\}$ is a $\sigma$-algebra. Hence every set $\cal B \times \cal C$ either contains $B \times C$ or is disjoint with $B \times C.$ But then the union of all such sets (atoms) is equal to $X \times Y.$ Thus $\cal B \otimes \cal C$ is atomic.
\end{proof}

\begin{thm} Let $(X, \cal A), (U, \cal F)$ and $(U, \cal G)$ be measure spaces. In the following, [(i) $\Rightarrow$ (ii)] is always true. The assertion [(ii) $\Rightarrow$ (iii)] is true if $\cal A$ and $\cal F \cap \cal G$ are atomic.  

(i) $$(\cal A \otimes \cal F) \cap (\cal A \otimes \cal G) = \cal A \otimes (\cal F \cap \cal G)$$
(ii)$$(\cal A \otimes \cal F) \cap (\cal A \otimes \cal G) \mbox{ and } \cal A \otimes (\cal F \cap \cal G) \mbox{ have the same atoms} $$
(iii)$$ \mbox{ All atoms of } (\cal A \otimes \cal F) \cap (\cal A \otimes \cal G) \mbox{ are cartesian products }$$

If $(\cal A \otimes \cal F) \cap (\cal A \otimes \cal G) $ is strongly Blackwell (or a sub $\sigma$-algebra of a strongly Blackwell space $(X, \cal M$)),  $(\cal A \otimes \cal F) \cap (\cal A \otimes \cal G) $ is countably generated and $\cal A \otimes (\cal F \cap \cal G)$ is countably generated, then (iii) implies (i).
\end{thm}

\begin{proof}
The proof is a direct consequence of Lemma \ref{lem:productAtoms} and the definition of a strongly Blackwell space. 
\end{proof}

\bibliographystyle{plain}

\end{document}